\documentclass[a4paper,11pt]{article}
\usepackage{amsmath,amsthm,amsfonts,amssymb,bbm,graphics,psfrag,amscd,color, epsfig}
\usepackage{fullpage}
\title{Fragmenting random permutations}

\date{December 4, 2007}

\makeatletter
\author{Christina Goldschmidt, James B.\ Martin and Dario Span\`o
  \thanks{\texttt{$\{$goldschm,martin,spano$\}$@stats.ox.ac.uk}}\\
\makeatother
\emph{Department of Statistics, University of Oxford}}

\newcommand{\E}[1]{\ensuremath{\mathbb{E} \left[#1 \right]}}
\newcommand{\Prob}[1]{\ensuremath{\mathbb{P} \left(#1 \right)}}

\newcommand{\I}[1]{\ensuremath{\mathbbm{1}_{ \{ #1 \} }}}
\newcommand{\R}{\ensuremath{\mathbb{R}}}
\newcommand{\Z}{\ensuremath{\mathbb{Z}}}
\newcommand{\N}{\ensuremath{\mathbb{N}}}

\renewcommand{\subset}{\subseteq}

\newcommand{\equidist}{\ensuremath{\stackrel{d}{=}}}

\newcommand{\cS}{\ensuremath{\mathcal{S}}}
\newcommand{\cP}{\ensuremath{\mathcal{P}}}

\newcommand{\bw}{\ensuremath{\mathbf{w}}}
\newcommand{\bv}{\ensuremath{\mathbf{v}}}

\newtheorem{thm}{Theorem}[section]

\newtheorem{lemma}[thm]{Lemma}
\newtheorem{proposition}[thm]{Lemma}
\newtheorem{prop}[thm]{Proposition}

\begin{document}
\maketitle

\begin{abstract}
  \noindent \textbf{Problem 1.5.7 from Pitman's Saint-Flour lecture
    notes~\cite{PitmanStFl}:} Does there exist for each $n$ a
  $\mathcal{P}_n$-valued fragmentation process $(\Pi_{n,k}, 1 \leq k
  \leq n)$ such that $\Pi_{n,k}$ is distributed like the partition
  generated by cycles of a uniform random permutation of $[n]$
  conditioned to have $k$ cycles? We show that the answer is yes.  We
  also give a partial extension to general exchangeable Gibbs
  partitions.
\end{abstract}

\section{Introduction}
Let $[n]=\{1,2,\dots, n\}$, let $\cP^n$ be the set of partitions of $[n]$,
and let $\cP^n_k$ be the set of partitions of $[n]$ into precisely $k$ blocks.

The main result of this paper concerns the partition of $[n]$ 
induced by the cycles of a uniform random permutation of $[n]$.
We begin by putting this in the context of more
general \emph{Gibbs partitions}.
Suppose that $v_{n,k}$, $n \geq 1$, $1 \leq k \leq n$ is a triangular array of
non-negative reals and $w_j$, $j \geq 1$ is a sequence of non-negative
reals. For given $n$, a random partition $\Pi$ is said to have the
Gibbs$(\bv,\bw)$ distribution on $\cP^n$ if, for any $1 \leq k \leq n$
and any partition $\{A_1, A_2,\dots, A_k\}\in\cP^n_k$, we have
\[
\Prob{\Pi=\{A_1,A_2,\dots, A_k\}} = v_{n,k} \prod_{j=1}^k w_{|A_j|}.
\]
In order for this to be a well-defined distribution, the weights
should satisfy the normalisation condition
\[
\sum_{k=1}^n v_{n,k} B_{n,k}(\bw) = 1,
\]
where
\[
B_{n,k}(\bw) := \sum_{\{A_1, A_2, \ldots, A_k\} \in \cP^n_k}
\prod_{j=1}^k w_{|A_j|}.
\]
$B_{n,k}(\bw)$ is a partial Bell polynomial in the variables $w_1,
w_2, \ldots$.  Let $K_n$ be the number of blocks of $\Pi$.  Then it is
straightforward to see that
\[
\Prob{K_n = k} = v_{n,k} B_{n,k}(\bw)
\]
and so the distribution of $\Pi$ conditioned on the event $\{K_n =
k\}$ does not depend on the weights $v_{n,k}$, $n \geq 1$, $1 \leq k
\leq n$:
\[
\Prob{\Pi=\{A_1,A_2,\dots, A_k\} | K_n=k} = \frac{\prod_{j=1}^{k}
  w_{|A_j|}}{B_{n,k}(\bw)}.
\]
By a Gibbs$(\bw)$ partition on $\cP^n_k$ we will mean a
Gibbs$(\bv,\bw)$ partition of $[n]$ conditioned to have $k$ blocks.

For given $n$, a Gibbs$(\bw)$ fragmentation process is then a process
$(\Pi_{n,k}, 1 \leq k \leq n)$ such that $\Pi_{n,k}\in\cP^n_k$ for all
$k$, which satisfies the following properties:
\begin{itemize}
\item[(i)]
for $k=1,2,\dots, n$, we have that $\Pi_{n,k}$ is a Gibbs($\bw$) partition;
\item[(ii)]
for $k=1,2,\dots, n-1$,
the partition $\Pi_{n,k+1}$ is obtained from the
partition $\Pi_{n,k}$ by splitting one of the blocks into two parts.
\end{itemize}

If $w_j=(j-1)!$, the Gibbs$(\bw)$ distribution on $\cP^n_k$ is the
distribution of a partition into blocks given by the cycles of a
uniform random permutation of $[n]$, conditioned to have $k$
cycles.  Problem 1.5.7 from Pitman's Saint-Flour lecture notes
\cite{PitmanStFl} asks whether a Gibbs fragmentation exists for these
weights; see also \cite{Berestycki/Pitman} for further discussion and
for results concerning several closely related questions.

In Section \ref{sec:existence} we will show that such a process does indeed exist.
Using the Chinese restaurant process construction of a random permutation,
we reduce the problem to one concerning sequences of independent Bernoulli random
variables, conditioned on their sum. In Section \ref{sec:recursive} we
describe a more explicit recursive construction of such a fragmentation process.
Finally in Section~\ref{sec:extension} we consider the
properties of Gibbs$(\bw)$ partitions with more general
weight sequences, corresponding to a class of exchangeable partitions of $\N$
(and including the two-parameter family of $(\alpha,\theta)$-partitions).
For these weight sequences we prove that, for fixed $n$,
one can couple partitions of $[n]$ conditioned to have $k$ blocks,
for $1\leq k\leq n$,
in such a way that the set of elements which
are the smallest in their block is increasing in $k$.
This extends the result from Section \ref{sec:existence} on Bernoulli random variables
conditioned on their sum; it is a necessary but not sufficient condition for
the existence of a fragmentation process.

\section{Existence of a fragmentation process}\label{sec:existence}

Consider the \emph{Chinese restaurant process} construction of a
uniform random permutation of $[n]$, due to Dubins and Pitman (see,
for example, Pitman~\cite{PitmanStFl}).  Customers arrive in the
restaurant one by one.  The first customer sits at the first table.
Customer $i$ chooses one of the $i-1$ places to the left of an
existing customer, or chooses to sit at a new table, each with
probability $1/i$.  So we can represent our uniform random permutation
as follows.  Let $C_2, \ldots, C_n$ be a sequence of independent
random variables such that $C_i$ is uniform on the set
${1,2,\ldots,i-1}$.  Let $B_1, B_2, \ldots, B_n$ be independent
Bernoulli random variables, independent of the sequence $C_2, C_3,
\ldots, C_n$ and with $B_i$ having mean $1/i$.  If $B_i = 1$ then
customer $i$ starts a new table.
Otherwise, $C_i$ gives the label of the customer whom he sits next to.
The state of the system after $n$ customers have arrived
describes a uniform random permutation of $[n]$;
each table in the restaurant corresponds to a cycle of the permutation,
and the order of customers around the table gives the order in the cycle.
Write $\Pi(B_1, B_2, \ldots, B_n, C_2, C_3, \ldots, C_n)$ for the random
partition generated in this way (the blocks of the partition
correspond to cycles in the permutation, i.e.\ to tables in the restaurant).

This construction has two particular features that will be important.
Firstly, the number of blocks in the partition is simply $\sum_{i=1}^n
B_i$.  So if we condition on $\sum_{i=1}^n B_i=k$, we obtain precisely
the desired distribution on $\cP^n_k$, of the partition obtained from
a uniform random permutation of $[n]$ conditioned to have $k$ cycles.
Secondly, if one changes one of the $B_i$ from 0 to 1 (hence
increasing the sum by 1), this results in one of the blocks of the
partition splitting into two parts.

Hence, we can use this representation to construct our sequence
of partitions $(\Pi_{n,k}, 1 \leq k \leq n)$.
We will need the following result.
\begin{prop}\label{couplingprop}
Let $n$ be fixed. Then there exists a coupling of random variables
$B_i^k$, $1\leq i\leq n$, $1\leq k\leq n$ with the following properties:
\begin{itemize}
\item[(i)]
for each $k$,
\[
(B_1^{k}, B_2^{k}, \ldots, B_n^{k}) \equidist \left(B_1, B_2,
  \ldots, B_n \Bigg| \sum_{i=1}^n B_i = k\right);
\]
\item[(ii)]
for all $k$ and $i$, if $B_i^k=1$ then $B_i^{k+1}=1$, with probability 1.
\end{itemize}
\end{prop}

So we fix $C_2, C_3, \ldots, C_n$, and define
\[
\Pi_{n,k} = \Pi(B_1^{k}, B_2^{k}, \ldots, B_n^{k}, C_2, C_3, \ldots, C_n).
\]
Then $(\Pi_{n,k}, 1 \leq k \leq n)$ is the desired Gibbs fragmentation process.

It remains to prove Proposition \ref{couplingprop}.
Write $B^{k} =
(B_1^{k}, B_2^{k}, \ldots, B_n^{k})$, $b = (b_1, b_2, \ldots, b_n)$ and
\[
p^{k}(b) := \Prob{B^{k,n} = b},
\]
where $b_i \in \{0,1\}$ for $1 \leq i \leq n$.  Clearly, $p^{k}(b)$ is
only non-zero for sequences $b$ having exactly $k$ $1$'s.  Write
$\cS_k$ for the subset of $\{0,1\}^n$ consisting of these sequences
and write $b \prec b'$ whenever $b'$ can be obtained from $b$ by
replacing one of the co-ordinates of $b$ which is $0$ by a $1$.
We need a process $(B^1, B^2, \dots, B^n)$ whose $k$th marginal
$B^k$ has distribution $p^k$ and such that, with probability 1,
$B^k\prec B^{k+1}$ for each $k$.

It is enough to show that for each $k$, we can couple $B^k$ with
distribution $p^k$ and $B^{k+1}$ with distribution $p^{k+1}$ in such a
way that $B^k\prec B^{k+1}$.  (The couplings can then be combined, for
example in a Markovian way, to give the desired law on the whole
sequence). As noted in Section 4 of Berestycki and
Pitman~\cite{Berestycki/Pitman}, a necessary and sufficient condition
for such couplings to exist with specified marginals and order
properties was given by Strassen; see for example Theorem 1 and
Proposition 4 of \cite{KKO}.  (Strassen's theorem may be seen as a
version of Hall's marriage theorem \cite{Hall} and is closely related
to the max-flow/min-cut theorem \cite{EFS,FF}).  The required
condition may be stated as follows.  For $C \subset \cS_k$, write
$N(C) = \{b' \in \cS_{k+1}: b \prec b' \text{ some $b \in C$}\}$.  We
then need that for all $C \subset \cS_k$,
\begin{equation}\label{Strassencondition}
\sum_{b \in C} p^k(b) \leq \sum_{b' \in N(C)} p^{k+1}(b')
\end{equation}

\begin{figure}
\begin{center}
\psfrag{a}{$\frac{6}{11}$}
\psfrag{b}{$\frac{3}{11}$}
\psfrag{c}{$\frac{2}{11}$}
\psfrag{e}{$\frac{3}{6}$}
\psfrag{f}{$\frac{2}{6}$}
\psfrag{g}{$\frac{1}{6}$}
\includegraphics{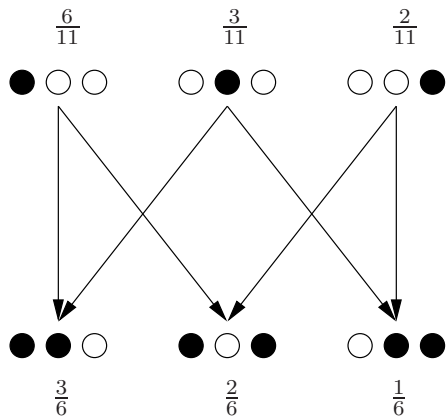}
\end{center}
\caption{The case $n = 4$ and $k=2,3$.
Since $B_1$ is always 1, we omit it from the picture.
For $2 \leq i \leq 4$, a
  filled circle in position $i$ indicates that $B_i^{k} = 1$; an empty
  circle indicates that it is $0$.  The arrows indicate which states
  with $k=3$ can be obtained from those with $k=2$; the numbers are the
  probabilities $p_k(b)$ of the states.}
\label{fig:simplest}
\end{figure}

Phrasing things a little differently, we need the following
proposition.

\begin{prop} \label{prop:cond}
Let $A_1, A_2, \ldots, A_m$ be any collection of distinct $k$-subsets of $[n]$ and let
\[
E_j = \{B_i = 1 \ \forall\ i \in A_j\},\ 1 \leq j \leq m.
\]
Then
\[
\Prob{E_1 \cup E_2 \cup \cdots \cup E_m \Bigg| \sum_{i=1}^n B_i = k}
\leq \Prob{E_1 \cup E_2 \cup \cdots \cup E_m \Bigg| \sum_{i=1}^n B_i = k+1},
\]
$1 \leq k \leq n-1$.
\end{prop}

This is a corollary of the following result from Efron~\cite{Efron}.

\begin{prop}
Let $\phi_n: \{0,1\}^n \to \R_+$ be a function which is increasing in
all of its arguments.  Let $I_i$, $1\leq i \leq n$, be independent
Bernoulli random variables (not necessarily with the same parameter).
Then
\[
\E{\phi_n(I_1, I_2, \ldots, I_n) \Bigg| \sum_{i=1}^n I_i = k}
\leq \E{\phi_n(I_1, I_2, \ldots, I_n) \Bigg| \sum_{i=1}^n I_i = k+1},
\]
for all $0 \leq k \leq n-1$.
\end{prop}

Since Efron's proof does not apply directly to the case of discrete
random variables, we give a proof here.

\begin{proof}
  First let $\psi: \Z_+^2 \to \R_+$ be a function which is increasing in
  both arguments.  Fix $n \geq 1$ and write $X_n = \sum_{i=1}^n I_i$.
  We will first prove that for $0 \leq k \leq n$,
\begin{equation}
\E{\psi(X_n, I_{n+1}) | X_n + I_{n+1} = k}
\leq \E{\psi(X_n, I_{n+1}) | X_n+I_{n+1} = k+1}. \label{eqn:m=2}
\end{equation}
Let $u(i) = \Prob{X_n = i}$ for $0 \leq i \leq n$.  We observe that,
as a sum of independent Bernoulli random variables,
$X_n$ is
\emph{log-concave}, that is
\begin{equation} \label{eqn:lc}
u(i)^2 \geq u(i-1)u(i+1), \quad i \geq 1.
\end{equation}
(This follows because any Bernoulli random variable is log-concave,
and the sum of independent log-concave random variables is itself
log-concave, as proved by Hoggar \cite{MatthewHoggard}).

Since $\psi(0,0) \leq \psi(1,0)$ and $\psi(0,0) \leq \psi(0,1)$, we
have
\begin{align*}
\E{\psi(X_n,I_{n+1}) | X_n + I_{n+1} = 0}
& = \psi(0,0) \\
& \leq  \frac{p_{n+1} u(0) \psi(1,0) + (1-p_{n+1}) u(1) \psi(0,1)}
             {p_{n+1} u(0) + (1 - p_{n+1})u(1)} \\
& = \E{\psi(X_n,I_{n+1}) | X_n + I_{n+1} = 1}.
\end{align*}
Similarly, since $\psi(n-1,1) \leq \psi(n,1)$ and $\psi(n,0) \leq
\psi(n,1)$, we have
\begin{align*}
\E{\psi(X_n, I_{n+1}) | X_n + I_{n+1} = n}
& = \frac{p_{n+1} u(n-1) \psi(n-1,1) + (1-p_{n+1}) u(n) \psi(n,0)}
         {p_{n+1} u(n-1) + (1-p_{n+1}) u(n)} \\
& \leq \psi(n,1) \\
& = \E{\psi(X_n, I_{n+1}) | X_n + I_{n+1} = n+1}.
\end{align*}
Suppose now that $1 \leq k \leq n-1$.  Then
\begin{align}
& \E{\psi(X_n, I_{n+1}) | X_n + I_{n+1} = k+1}
- \E{\psi(X_n, I_{n+1}) | X_n + I_{n+1} = k} \notag \\
& = \frac{p_{n+1} u(k) \psi(k,1) + (1-p_{n+1}) u(k+1) \psi(k+1,0)}
         {p_{n+1} u(k) + (1-p_{n+1}) u(k+1)}  \notag \\
& \qquad  - \frac{p_{n+1} u(k-1) \psi(k-1,1) + (1-p_{n+1}) u(k) \psi(k,0)}
         {p_{n+1} u(k-1) + (1-p_{n+1}) u(k)} \notag \\
& = \frac{p_{n+1}^2 u(k)u(k-1) [\psi(k,1) - \psi(k-1,1)]}{d} \notag \\
& \qquad
+ \frac{(1-p_{n+1})^2u(k) u(k+1) [\psi(k+1,0) - \psi(k,0)]}{d}
\notag \\
& \qquad + \frac{p_{n+1}(1-p_{n+1})
                 [u(k-1) u(k+1) \psi(k+1,0) - u(k)^2 \psi(k,0)]}{d}
\notag \\
& \qquad + \frac{p_{n+1}(1-p_{n+1} )
                 [u(k)^2 \psi(k,1) - u(k-1) u(k+1) \psi(k-1,1)]}{d},
\label{eqn:sum}
\end{align}
where the denominator $d$ is given by
\[
d = \big[p_{n+1} u(k) + (1-p_{n+1}) u(k+1)\big]\big[ p_{n+1} u(k-1) + (1-p_{n+1}) u(k)\big]
\]
and is clearly non-negative.  The first two terms in (\ref{eqn:sum})
are non-negative because $\psi$ is increasing.  The sum of the third
and fourth terms is bounded below by
\[
\frac{p_{n+1}(1-p_{n+1})}{d}
[u(k)^2 - u(k-1)u(k+1)] [\psi(k,1) - \psi(k,0)].
\]
This is non-negative by the log-concavity property (\ref{eqn:lc}).
Putting all of this together, we see that
\[
\E{\psi(X_n, I_{n+1}) | X_n + I_{n+1} = k}
\leq \E{\psi(X_n, I_{n+1}) | X_n+I_{n+1} = k+1},
\]
as required.

We now proceed by induction on $n$.  Note that (\ref{eqn:m=2}) with
$n=1$ gives the base case. Now define
\[
\psi_n\left(\sum_{i=1}^{n-1} I_i, I_n\right) = \E{\phi_n(I_1, I_2,
  \ldots, I_n) \Bigg| \sum_{i=1}^{n-1} I_i, I_n}.
\]
Assume that we have
\[
\E{\phi_{n-1}(I_1, I_2, \ldots, I_{n-1}) \Bigg| \sum_{i=1}^{n-1} I_i = k}
\leq \E{\phi_{n-1}(I_1, I_2, \ldots, I_{n-1}) \Bigg| \sum_{i=1}^{n-1} I_i = k+1}
\]
for $0 \leq k \leq n-2$.  By this induction hypothesis,
$\psi_n\left(\sum_{i=1}^{n-1} I_i, I_n\right)$ is increasing in its
first argument.  By the assumption that $\phi_n$ is increasing,
$\psi_n\left(\sum_{i=1}^{n-1} I_i, I_n\right)$ is increasing in its
second argument.  So by (\ref{eqn:m=2}),
\[
\E{\psi_n\left(\sum_{i=1}^{n-1} I_i, I_n\right) \Bigg|
  \sum_{i=1}^{n-1}I_i + I_n = k}
\leq \E{\psi_n\left(\sum_{i=1}^{n-1} I_i, I_n\right) \Bigg|
  \sum_{i=1}^{n-1}I_i + I_n = k+1}
\]
for $0 \leq k \leq n-1$.  But by the tower law, this says exactly that
\[
\E{\phi_n(I_1, I_2, \ldots, I_n) \Bigg| \sum_{i=1}^n I_i = k}
\leq \E{\phi_n(I_1, I_2, \ldots, I_n) \Bigg| \sum_{i=1}^n I_i = k+1}.
\]
The result follows by induction.
\end{proof}

\emph{Proof of Proposition~\ref{prop:cond}}.
Let $\phi_n(B_1, B_2, \ldots, B_n) = \I{E_1 \cup E_2 \cup \cdots \cup
  E_m}$.  Then $\phi_n$ is increasing in all of its arguments and so
by the previous proposition we have
\begin{align*}
\Prob{E_1 \cup E_2 \cup \cdots \cup E_m \Bigg| \sum_{i=1}^n B_i = k}
& = \E{\phi_n (X_1, X_2, \ldots, X_n) \Bigg| \sum_{i=1}^n B_i = k} \\
& \leq \E{\phi_n (X_1, X_2, \ldots, X_n) \Bigg| \sum_{i=1}^n B_i = k+1}\\
& = \Prob{E_1 \cup E_2 \cup \cdots \cup E_m \Bigg| \sum_{i=1}^n B_i =
  k+1}.
\end{align*}\hfill $\Box$

Note that we have only proved the existence of a coupling.  There is
no reason why it should be unique.  Indeed, in general, there is a
simplex of solutions.  For the example given in
Figure~\ref{fig:simplest}, the extremes of the one-parameter family of
solutions are shown in Figure~\ref{fig:2solns}.

\begin{figure}
\begin{center}
\psfrag{a}{$\frac{6}{11}$}
\psfrag{b}{$\frac{3}{11}$}
\psfrag{c}{$\frac{2}{11}$}
\psfrag{e}{$\frac{3}{6}$}
\psfrag{f}{$\frac{2}{6}$}
\psfrag{g}{$\frac{1}{6}$}
\psfrag{d}{$\frac{15}{66}$}
\psfrag{i}{$\frac{21}{66}$}
\psfrag{l}{$\frac{18}{66}$}
\psfrag{m}{$0$}
\psfrag{o}{$\frac{1}{66}$}
\psfrag{r}{$\frac{11}{66}$}
\psfrag{h}{$\frac{26}{66}$}
\psfrag{j}{$\frac{10}{66}$}
\psfrag{k}{$\frac{7}{66}$}
\psfrag{n}{$\frac{11}{66}$}
\psfrag{p}{$\frac{12}{66}$}
\psfrag{q}{$0$}
\includegraphics{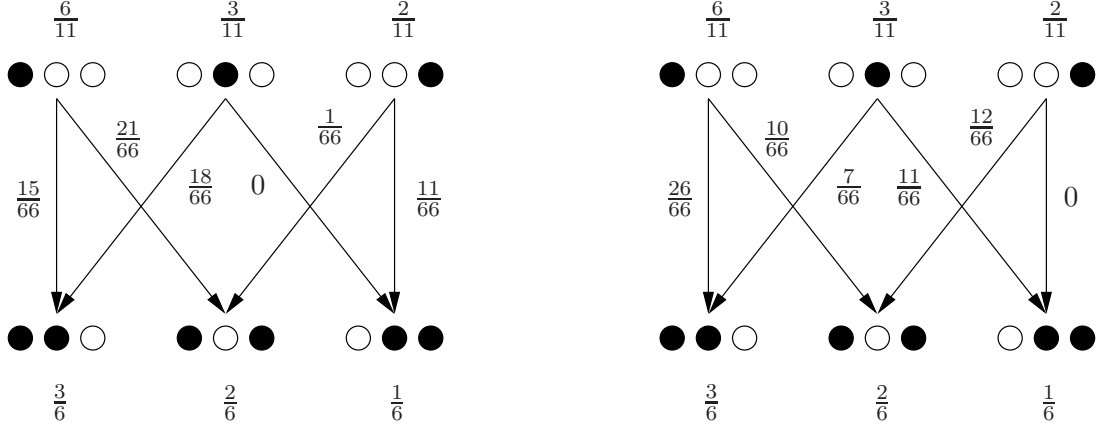}
\end{center}
\caption{The two extreme solutions to the coupling problem for $n=4$
  and $k=2,3$ (as drawn in Figure \ref{fig:simplest}).
The arrows are labelled with the joint probabilities
$\Prob{B^k=b,B^{k+1}=b'}$.}
\label{fig:2solns}
\end{figure}

\section{A recursive construction}
\label{sec:recursive}

In this section, we describe a more explicit construction of the Gibbs
fragmentation processes, which is recursive in $n$ and possesses a
certain consistency property as $n$ varies.

The basic principle is the following simple observation.  Suppose we
want to create a uniform random permutation of $[n]$ conditioned to
have $k$ cycles.  Then $n$ either forms a singleton, or is contained
in some cycle with other individuals.  If it forms a singleton, then
the rest of the permutation is a uniform random permutation of
$[n-1]$, conditioned to have $k-1$ cycles.  If, on the other hand, $n$
is not a singleton, then we take a uniform random permutation of
$[n-1]$ into $k$ cycles and insert $n$ into a uniformly chosen
position.

So we proceed as follows.  A Gibbs fragmentation process on $[n]$ for
$n=1$ or $n=2$ is trivial.

Suppose we have constructed a process $(\Pi^{n-1}_1,\Pi^{n-1}_2,\dots,
\Pi^{n-1}_{n-1})$ on $[n-1]$, with the required marginal distributions
and splitting properties.  We will derive from it a process
$(\Pi^{n}_1, \Pi^{n}_2,\dots, \Pi^{n}_n)$ on $[n]$.

For each $k$, the partition $\Pi_k^n$ of $[n]$ into $k$ parts will
come from either
\begin{itemize}
\item[(a)]
adding a singleton block $\{n\}$ to $\Pi_{k-1}^{n-1}$; or
\item[(b)]
adding the element $n$ to one of the blocks of $\Pi_{k}^{n-1}$,
by choosing an element $C_{n}$ uniformly at random from $[n-1]$ and
putting $n$ in the same block as $C_n$.
\end{itemize}

Note that
\[
\Prob{\{n\} \text{ is a singleton in }\Pi_k^{n}}
=
\Prob{B_n=1 \bigg| \sum_{i=1}^n B_i=k}
\]
which is increasing in $k$, by the monotonicity results proved in the previous section.
So let $K_n$ be a random variable whose distribution is given by
\[
\Prob{K_n\leq k}=\Prob{\{n\} \text{ is a singleton in }\Pi_k^{n}};
\]
now we take choice (a) above for all $k\geq K_n$, and choice (b)
for all $k<K_n$ (using the same value of $C_n$ for each such $k$).

A possible realisation of this recursive process is
illustrated in Figure~\ref{fig:table}.

\begin{figure}
\centering
\epsfig{figure=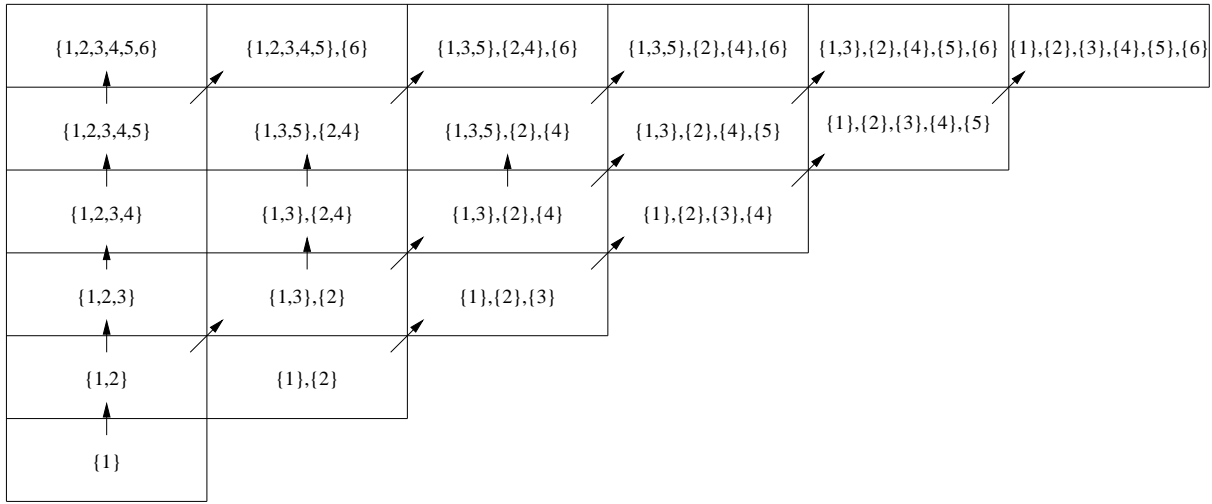, width=\linewidth}
\caption{A possible realisation of the coupling.}
\label{fig:table}
\end{figure}

\section{Partial extension to exchangeable Gibbs partitions}
\label{sec:extension}

At the beginning of this paper, we introduced the notion of a
partition of $[n]$ with Gibbs$(\bv,\bw)$ distribution.  In general,
there is no reason why these partitions should be consistent as $n$
varies.  That is, it is not necessarily the case that taking a
Gibbs$(\bv,\bw)$ partition of $[n+1]$ and deleting $n+1$ gives rise to a
Gibbs$(\bv,\bw)$ partition of $[n]$.  If, however, this \emph{is} the case,
we can define an exchangeable random partition $\Pi$ of $\N$ as the
limit of the projections onto $[n]$ as $n \to \infty$.  In this case,
we refer to $\Pi$ as an \emph{exchangeable Gibbs$(\bv,\bw)$
  partition}.

An important subfamily of the exchangeable Gibbs$(\bv,\bw)$ partitions
are the $(\alpha,\theta)$-partitions, whose asymptotic frequencies
have the Poisson-Dirichlet$(\alpha, \theta)$ distribution, for $0 \leq
\alpha < 1$ and $\theta > -\alpha$, or $\alpha < 0$ and $\theta =
m|\alpha|$, some $m \in \N$ (see Pitman and Yor~\cite{Pitman/Yor} or
Pitman~\cite{PitmanStFl} for a wealth of information about these
distributions).  Here, the corresponding weight sequences are
\[
w_j = (1-\alpha)_{j-1 \uparrow 1}, \quad v_{n,k} =
\frac{(\theta+\alpha)_{(k-1)\uparrow \alpha}}{(\theta + 1)_{(n-1)
    \uparrow 1}},
\]
where $(x)_{m \uparrow \beta} := \prod_{j=1}^{m} (x + (j-1) \beta)$
and $(x)_{0 \uparrow \beta} := 1$.  In the first part of this paper,
we have treated the case $\alpha = 0$: a $(0,1)$-partition of $[n]$
has the same distribution as the partition derived from the cycles of
a uniform random permutation.  In view of the fact that the array
$\bv$ does not influence the partition conditioned to have $k$ blocks,
we have also proved that a fragmentation process having the
distribution at time $k$ of a $(0,\theta)$-partition conditioned to
have $k$ blocks exists.  The Gibbs$(\bw)$ distribution corresponding
to the case $\alpha = -1$ (i.e.\ to weight sequence $w_j = j!$) is
discussed in Berestycki and Pitman~\cite{Berestycki/Pitman}; in
particular, it is known that a Gibbs$(\bw)$ fragmentation exists in
this case.

Gnedin and Pitman~\cite{Gnedin/Pitman} have proved that the
exchangeable Gibbs$(\bv,\bw)$ partitions all have $\bw$-sequences of
the form
\[
w_j = (1-\alpha)_{j-1 \uparrow 1}
\]
for some $-\infty \leq \alpha \leq 1$, where for $\alpha = -\infty$
the weight sequence is interpreted as being identically equal to 1 for
all $j$.  Furthermore, the array $\bv$ must solve the backward
recursion
\[
v_{n,k} = \gamma_{n,k} v_{n+1,k} + v_{n+1,k+1}, \quad 1 \leq k \leq n,
\]
where $v_{1,1} = 1$ and
\begin{equation}\label{gammadef}
\gamma_{n,k} =
\begin{cases}
n-\alpha k & \text{if $-\infty < \alpha < 1$} \\
k & \text{if $\alpha = -\infty$}.
\end{cases}
\end{equation}
The case $\alpha = 1$ corresponds to the trivial partition
into singletons and will not be discussed any further.

It seems natural to consider the question of whether Gibbs
fragmentations exist for other weight sequences falling into the
exchangeable class.  We will give here a partial extension of our
results to the case of a general $\alpha \in [-\infty,1)$.

Fix $n$ and consider a Gibbs$(\bv,\bw)$ partition of $[n]$ with $w_j =
(1 - \alpha)_{j-1 \uparrow 1}$ for some $\alpha \in [-\infty,1)$.  Let
$B_i$ be the indicator function of the event that $i$ is the smallest
element in its block, for $1 \leq i \leq n$, and let $K_n =
\sum_{i=1}^{n} B_i$, the number of blocks.  Then we have the following
extension of the earlier Proposition \ref{couplingprop}.

\begin{prop}
\label{alphathetaprop}
For each $n$, there exists a
random sequence of vectors
$\left(B_1^{k,n}, \dots, B_n^{k,n}\right)$, $k=1,2,\dots,n$
such that for each $k$,
\[
\left(B_1^{k,n}, \dots, B_n^{k,n}\right)
\equidist
\left(B_1,\dots, B_n | K_n=k\right),
\]
and such that $B_i^{n,k+1}\geq B_i^{n,k}$ for all $k$.
\end{prop}

This is precisely what we proved earlier for the case $\alpha=0$ (and,
indeed, for more general sequences of independent Bernoulli random
variables).
Namely we show that Gibbs$(\bw)$ partitions of $[n]$, conditioned to have
$k$ blocks, can be coupled over $k$ in such a way that the set of elements which
are the smallest in their block is increasing in $k$.
In the context of random permutations, this
was enough to prove the existence of a coupling of
\textit{partitions} with the desired fragmentation property, using the
fact that the random variables $C_i$ in the Chinese restaurant
process, which govern the table joined by each arriving customer, were
independent of each other and from the random variables $B_i$.  For
$\alpha \neq 0$, this is no longer the case 
and we cannot deduce the full result for partitions from
this result for ``increments'' or ``records''.
In fact, this stronger result is not always true, for example when $\alpha=-\infty$
and so $w_j=1$ for all $j$.
In this case it is known that
no Gibbs fragmentation process exists for $n=20$
and for all large enough $n$ (see for example \cite{Berestycki/Pitman} for a discussion).
By a continuity argument, one can show similarly that if
$-\alpha$ is sufficiently large, then for certain $n$
no Gibbs fragmentation process exists.

As in Section \ref{sec:recursive},
Proposition \ref{alphathetaprop} can be proved by
induction over $n$. Suppose we have carried out the construction
for $n-1$, and wish to extend to $n$.
Conditioned on $K_n=k$, we need to consider two cases:
either $B_n=1$ and $K_{n-1}=k-1$, or $B_n=0$ and $K_n=k$.
Depending on which of these cases we choose, we will set either
\begin{align*}
\left(B_1^{n,k},\dots,B_n^{n,k}\right)
&=
\left(B_1^{n-1,k-1},\dots,B_{n-1}^{n-1,k-1},1\right)
\\
\intertext{or}
\left(B_1^{n,k},\dots,B_n^{n,k}\right)
&=
\left(B_1^{n-1,k},\dots,B_{n-1}^{n-1,k},0\right).
\end{align*}

Precisely as in Section \ref{sec:recursive},
this can be made to work successfully provided that the following lemma holds.

\begin{lemma}
\label{monotonicitylemma}
For any $n$, the probability $\Prob{B_n=1|K_n=k}$
is increasing in $k$.
\end{lemma}

Equivalently, we are showing that the expected number of singletons
in $\Pi_n$, conditioned on $K_n=k$, is increasing in $k$
(since, by exchangability, the probability that $n$ is a singleton
in $\Pi_n$ is the same as the probability that $i$ is a singleton,
for any $i\in[n]$).

The rest of this section is devoted to the proof of this lemma.

Define $S_\alpha(n,k) = B_{n,k}(\bw)$.  Then
\begin{equation}
\label{Sdef}
\Prob{K_n=k}= v_{n,k} S_\alpha(n,k),
\end{equation}
where the $S_\alpha$ obey the recursion
\begin{equation}
\label{Srec}
S_\alpha(n+1,k)= \gamma_{n,k} S_\alpha(n,k) + S_\alpha(n,k-1),
\end{equation}
with boundary conditions $S_\alpha(1,1)=1$, $S_\alpha(n,0)=0$ for all
$n$, and $S_\alpha(n,n+1)$=0 for all $n$.  These are \emph{generalized
  Stirling numbers}.  We have
\[
\Prob{B_n=1|K_n=k} = \Prob{K_{n-1} = k-1 | K_n = k} =
\frac{S_{\alpha}(n-1,k-1)}{S_{\alpha}(n,k)}
\]
(see Section 3 of \cite{Gnedin/Pitman} for further details).
Using the recursion, this is equal to
\[
\left( 1 + \frac{ \gamma_{n-1,k} S_{\alpha}(n-1,k)}{S_{\alpha}(n-1,k-1)}
\right)^{-1}.
\]
For this to be increasing in $k$, it is equivalent that
\[
\frac{S_{\alpha}(n-1,k-1)}{\gamma_{n-1,k} S_{\alpha}(n-1,k)}
\]
should be increasing in $k$.  This is implied by the following
proposition.

\begin{proposition}
For all $\alpha < 1$ and all $n$ and $k$,
\[
\gamma_{n,k} S_\alpha(n,k)^2 \geq \gamma_{n,k+1} S_\alpha(n,k+1)S_\alpha(n,k-1).
\]
\end{proposition}

\begin{proof}
  If $0 \leq \alpha < 1$ then it is sufficient to prove the statement
  that $(S_{\alpha}(n,k), 0 \leq k \leq n)$ is log-concave, that is
\[
S_{\alpha}(n,k)^2 \geq S_{\alpha}(n,k-1) S_{\alpha}(n,k+1),
\]
since in that case $\gamma_{n,k}$, defined at (\ref{gammadef}), is decreasing in $k$. 
Theorem 1 of Sagan~\cite{Sagan} states that whenever $t_{n,k}$ is a
triangular array satisfying
\[
t_{n,k} = c_{n,k} t_{n-1,k-1} + d_{n,k} t_{n-1,k}
\]
for all $n \geq 1$, where $t_{n,k}$, $c_{n,k}$ and $d_{n,k}$ are all
integers and such that
\begin{itemize}
\item $c_{n,k}$ and $d_{n,k}$ are log-concave in $k$,
\item $c_{n,k-1} d_{n,k+1} + c_{n,k+1} d_{n,k-1} \leq 2 c_{n,k}
  d_{n,k}$ for all $n \geq 1$,
\end{itemize}
then $t_{n,k}$ is log-concave in $k$.  These conditions are clearly
satisfied for the generalized Stirling numbers $S_{\alpha}(n,k)$, with the
exception that the sequence $d_{n,k}$ is not integer-valued.  However,
Sagan's argument extends immediately to this case also, and so we will
not give a proof here.

We turn now to the case $-\infty < \alpha < 0$.  We proceed by
induction on $n$.  For $n=2$ the statement is trivial.  Suppose that
we have
\[
(n - \alpha k) S_{\alpha}(n,k)^2 \geq (n- \alpha(k+1)) S_{\alpha}(n,k-1) S_{\alpha}(n,k+1),
\]
for $0 \leq k \leq n$.  Using the recurrence
\[
S_{\alpha}(n+1,k) = (n - \alpha k) S_{\alpha}(n,k) + S_{\alpha}(n,k-1),
\]
we obtain that
\begin{align}
& (n+1 - \alpha k) S_{\alpha}(n+1,k)^2 - (n+1 - \alpha (k+1)) S_{\alpha}(n+1,
k-1)S_{\alpha}(n+1,k+1) \notag \\
& = \Big\{ (n+1 - \alpha k) (n- \alpha k)^2 S_{\alpha}(n,k)^2 \notag \\
& \qquad  - (n+1 -\alpha(k+1))(n-\alpha(k-1)) (n- \alpha (k+1))
  S_{\alpha}(n,k-1)S_{\alpha}(n,k+1) \Big\} \notag\\
& \quad + \Big\{ 2 (n+1 - \alpha k) (n-\alpha k) S_{\alpha}(n,k) S_{\alpha}(n,k-1)
 \notag \\
& \qquad - (n+1 - \alpha(k+1))(n-\alpha(k+1)) S_{\alpha}(n,k+1) S_{\alpha}(n,k-2)
   \notag \\
& \qquad  - (n+1 - \alpha(k+1))(n-\alpha(k-1)) S_{\alpha}(n,k) S_{\alpha}(n,k-1) \Big\}
   \notag \\
& + \Big\{ (n+1 - \alpha k) S_{\alpha}(n,k-1)^2 - (n+1 - \alpha(k+1)) S_{\alpha}(n,k-2)
  S_{\alpha}(n,k) \Big\}. \label{eqn:bigsum}
\end{align}
We take each of the three terms in braces separately.  For the first term,
we note that
\[
(n+1 - \alpha k)(n-\alpha k) - (n+1 - \alpha(k+1))(n-\alpha(k-1))
= \alpha(\alpha - 1).
\]
Hence, by the induction hypothesis, the first term is greater than or
equal to
\[
\alpha(\alpha - 1) (n-k\alpha) S_{\alpha}(n,k)^2.
\]
Since $\alpha < 0$, this is non-negative.  For the second term in
(\ref{eqn:bigsum}), we note that applying the induction hypothesis
twice entails that
\[
(n-\alpha(k-1)) S_{\alpha}(n,k)S_{\alpha}(n,k-1) \geq (n-\alpha(k+1)) S_{\alpha}(n,k+1) S_{\alpha}(n,k-2).
\]
So
\begin{align*}
& 2(n+1 - \alpha(k+1))(n - \alpha(k-1)) S_{\alpha}(n,k) S_{\alpha}(n,k-1) \\
& \geq (n+1 - \alpha(k+1))(n-\alpha(k+1)) S_{\alpha}(n,k+1) S_{\alpha}(n,k-2)\\
& \qquad + (n+1 - \alpha(k+1)) (n- \alpha(k-1)) S_{\alpha}(n,k) S_{\alpha}(n,k-1).
\end{align*}
It follows that the second term in braces is bounded below by
\begin{align*}
& 2 S_{\alpha}(n,k) S_{\alpha}(n,k-1) \left[ (n+1 - \alpha k)(n-\alpha k) - (n+1 -
  \alpha(k+1))(n - \alpha(k-1)) \right]  \\
& \qquad = 2 \alpha (\alpha - 1) S_{\alpha}(n,k) S_{\alpha}(n,k-1).
\end{align*}
This is non-negative.  Finally, we turn to the third term in braces in
(\ref{eqn:bigsum}).  This is equal to
\begin{align*}
& [(n - \alpha (k-1)) S_{\alpha}(n,k-1)^2 - (n - \alpha k)
S_{\alpha}(n,k-2) S_{\alpha}(n,k)] \\
& - (\alpha - 1) [S_{\alpha}(n,k-1)^2 - S_{\alpha}(n,k-2) S_{\alpha}(n,k)].
\end{align*}
By the induction hypothesis (at $n$ and $k-1$), the first term is
non-negative and so this quantity is bounded below by
\[
(1 - \alpha) [S_{\alpha}(n,k-1)^2 - S_{\alpha}(n,k-2) S_{\alpha}(n,k)].
\]
That $S_{\alpha}(n,k-1)^2 - S_{\alpha}(n,k-2) S_{\alpha}(n,k) \geq 0$
is implied by the induction hypothesis (in fact it is a weaker
statement) and so, since $\alpha < 0$, the above quantity is
non-negative.  So (\ref{eqn:bigsum}) is non-negative, as required.
The case $\alpha = -\infty$ follows similarly.
\end{proof}

\section*{Acknowledgments}
We are grateful to Jay Taylor and Amandine V\'eber for organising the
reading group on Pitman's Saint-Flour notes \cite{PitmanStFl} which
led to the work in this paper.  We would like to thank Sasha Gnedin
and Nathana\"el Berestycki for valuable discussions.  C.~G.\ was
funded by EPSRC Postdoctoral Fellowship EP/D065755/1. D.~S.\ was
funded in part by EPSRC grant GR/T21783/01.

\bibliography{permutation}

\end{document}